\documentclass[11pt]{amsart}
\usepackage[active]{srcltx}
\usepackage{amsmath, amsfonts,amsthm,times,graphics,color}
\newcommand{\vertiii}[1]{{\left\vert\kern-0.25ex\left\vert\kern-0.25ex\left\vert #1
    \right\vert\kern-0.25ex\right\vert\kern-0.25ex\right\vert}}
 \makeatletter
\renewcommand*\subjclass[2][2000]{%
  \def\@subjclass{#2}%
  \@ifundefined{subjclassname@#1}{%
    \ClassWarning{\@classname}{Unknown edition (#1) of Mathematics
      Subject Classification; using '1991'.}%
  }{%
    \@xp\let\@xp\subjclassname\csname subjclassname@#1\endcsname
  }%
}
 \makeatother
\usepackage{enumerate,amssymb,  mathrsfs,yhmath}%, pdfsync}% ,showkeys, pdfsync}

\newtheorem{theorem}{Theorem}[section]
\newtheorem{lemma}[theorem]{Lemma}
\newtheorem*{lemma*}{Lemma}
\newtheorem{proposition}[theorem]{Proposition}

\def\IB{{\Bbb B}}

\def\1ton{1,2,\ldots,n}
\def\det{{\rm det}}

\usepackage{amssymb}%, pdfsync}
\usepackage{amsthm}
\usepackage{mathrsfs, amsfonts, amsmath}
\usepackage{graphicx}
\usepackage[margin=3.3cm]{geometry}

\newcommand{\R}{\mathbb{R}}

\newcommand{\B}{\mathbb{B}}

\theoremstyle{definition}
\newtheorem{definition}[theorem]{Definition}

\theoremstyle{remark}
\newtheorem{remark}[theorem]{Remark}

\numberwithin{equation}{section}

\newcommand{\C}{\mathbb{C}}

\renewcommand{\imath}{i} %??? You already used $i$ for the imaginary unit (after (1.2)), which is the standard practice in math, I think -- I have not seen such usage of $\imath$ elsewhere. So, I think, if you insist on such usage of $\imath$, it should consitent throughout and also explained

\def\XXint#1#2#3{{\setbox0=\hbox{$#1{#2#3}{\int}$}
\vcenter{\hbox{$#2#3$}}\kern-.5\wd0}}

\def\ge{\geqslant}
\begin{document}

\title[A Faber-Krahn inequality for log-subharmonic functions in the ball]{A Faber-Krahn type inequality for log-subharmonic functions in the hyperbolic ball}

%\date{11 October, 2005}

\keywords{Hyperbolic harmonic functions, isoperimetric inequality, Bergman spaces}

\author{David Kalaj}
\address{University of Montenegro, Faculty of Natural Sciences and
Mathematics, Cetinjski put b.b. 81000 Podgorica, Montenegro}
\email{davidk@ucg.ac.me}

\author{Jo\~ao P. G. Ramos}
\address{ETH Z\"urich, D-MATH, R\"amistrasse 101, 8092 Z\"urich, Switzerland}
\email{joao.ramos@math.ethz.ch}

\begin{abstract}
Assume that $\Delta_h$ is the hyperbolic Laplacian in the unit ball $\mathbb{B}$ and assume that $\Phi_n$ is the unique radial solution of Poisson equation $\Delta_h \log \Phi_n =-4 (n-1)^2$ satisfying the condition $\Phi_n(0)=1$ and $\Phi_n(\zeta)=0$ for $\zeta\in \partial\B$. We explicitly solve the question of maximizing
$$
R_n(f,\Omega)=
\frac{\int_\Omega  |f(x)|^2 \Phi_n^\alpha(|x|) \, d\tau(x)}{\|f\|^2_{\mathbf{B}^2_\alpha}},
$$
over all $f \in\mathbf{B}^2_\alpha$ and $\Omega \subset \mathbb{B}$ with $\tau(\Omega) = s,$ where $d\tau$ denotes the invariant measure on $\B,$ and $\|f\|_{{B}^2_\alpha}^2 = \int_\mathbb{B} |f(x)|^2 \Phi_n^\alpha(|x|) d\tau(x) < \infty.$

This result extends the main result of Tilli and the second author \cite{ramostilli} to a higher-dimensional context. Our proof relies on a version of the techniques used for the two-dimensional case, with several additional technical difficulties arising from the definition of the weights $\Phi_n$ through hypergeometric functions.  Additionally, we show that an immediate relationship between a concentration result for log-sunharmonic functions and one for the Wavelet transform is only available in dimension one.
\end{abstract}
\maketitle
\tableofcontents
\sloppy

\maketitle
\section{Introduction}
Let $\mathbb{D}$ be the unit disk, $\alpha>1$ and let $\mathcal{A}_\alpha$ be the Bergman space of holomorphic functions defined on the unit disk so  that $\|f\|_{2,\alpha}^2:=\int_{\mathbb{D}}  |f(z)|^2 (1-|z|^2)^\alpha \, d\tau(z)< +\infty$. A fundamental question, raised by L. D. Abreu and M. D\"orfler in \cite{AbreuDoerfler} in connection to the optimal concentration for Wavelet transforms, is: for a fixed domain $\Omega \subset \mathbb{D},$ what is the maximum value that the quantity
$$\int_\Omega  |f(x)|^2 (1-|z|^2)^\alpha \, d\tau(z)$$
can achieve, where $d\tau=2(1-|z|^2)^{-2}dxdy$, $\tau(\Omega)=s$ and $\|f\|_{2,\alpha}=1$? Using a suitable version of techniques from \cite{3}, adapted and expanded to the hyperbolic case, the second author and P. Tilli \cite{ramostilli} were able to solve this problem \emph{exactly} for $\alpha > 1,$ thus solving the original question of characterizing sets of optimal concentration for certain special families of wavelet transforms.

Using a suitable version of techniques from \cite{3}, adapted and expanded to the hyperbolic case, the second author and P. Tilli \cite{ramostilli} were able to solve this problem \emph{exactly} for $\alpha > 1,$ thus solving the original question of characterizing sets of optimal concentration for certain special families of wavelet transforms. The techniques from \cite{3} were also used by A. Kulikov in \cite{kulik} to prove a  Wehrl-type entropy conjecture (see, for instance, \cite{9}) on the $SU(1,1)$ group, which is equivalent to a conjecture of Pavlovi\'c \cite{13} and a conjecture of O. Brevig, J. Ortega-Cerd\`a, K. Seip, and J. Zhao \cite{7} concerning certain embedding estimates for analytic functions. The fist author in \cite{kalaj} extended the result of A. Kulikov \cite{kulik}  to the higher-dimensional setting. For an extension of results from \cite{3, kulik}  we also refer to the recent paper by R. Frank \cite{rupert}. In the same spirit of those results, we finally also point to a recent partial solution to a contraction conjecture \cite{Jason} by P. Melentijevi\'c \cite{petar}, using the same circle of ideas. 

In this note, we extend the result of second author and P. Tilli \cite{ramostilli}  to the higher-dimensional setting. In order to do so, we first discuss the basic setup of the problem in the following section, introducing the relevant notions of operators and admissible spaces in the higher-dimensional context. In the third section, we prove our main result, by tailoring the general outline by the second author and P. Tilli \cite{ramostilli} to the case at hand; and in the last section, we briefly show that, although the class of special windows for which the Wavelet transform becomes analytic is non-empty in dimension one, in higher dimensions, when one restricts to \emph{radial} windows, one cannot simultaneously have hyperbolic harmonicity and log-subharmonicity, hinting at the fact that that case is fundamentally different, and must, as such, depart from analytic methods.

\section{Preliminaries}
\subsection{Harmonic maps}
We start out by recalling basic facts about Harmonic maps. The harmonic map equations for $u=(u^1,\dots
u^n):\mathcal N\to \mathcal M$ from the Riemann manifold $\mathcal
N=(B^n,(h_{jk})_{j,k})$ into a Riemann manifold $\mathcal
M=(\Omega,(g_{jk})_{j,k})$ (where $\Omega\subset \Bbb R^n$) are
\begin{equation}\label{cron} |h|^{-1/2}\sum_{\alpha,\beta=1}^n\partial_\alpha(|h|^{1/2}
h^{\alpha \beta}\partial_\beta u^i) + \sum_{\alpha, \beta, k, \ell =
1}^n\Gamma^i_{k\ell}(u)D_\alpha u^k D_\beta u^\ell = 0, \,\, i =
1,\dots,n,\end{equation} where $\Gamma^{i}_{k\ell}$ are Christoffel
Symbols of the metric tensor $(g_{jk})_{j,k}$ in the target space
$\mathcal M$:
$$    \Gamma^{i}_{k\ell}=\frac{1}{2}g^{im} \left(\frac{\partial
g_{mk}}{\partial x^\ell} +
    \frac{\partial g_{m\ell}}{\partial x^k} - \frac{\partial g_{k\ell}}{\partial x^m} \right) =
     {1 \over 2} g^{im} (g_{mk,\ell} + g_{m\ell,k} - g_{k\ell,m}),$$
the matrix $(g^{jk})_{j,k}$ (resp., $(h^{jk})_{j,k}$) is an inverse of the metric
tensor $(g_{jk})_{j,k}$ (resp. $(h_{jk})_{j,k}$), and $|h|=\det(h_{jk})_{j,k}$. See
e.g. \cite{jost} for further details.

Here and in the sequel we shall denote $|x|:=\sqrt{\sum_{k=1}^n x_k^2}$. We denote by $\mathcal{M}$ the group of M\"obius transformations of the unit ball onto itself. The M\"{o}bius invariant hyperbolic metric on the unit ball $\mathbb{B}$ is then defined by \begin{equation}\label{hm}h_{ij}(x)=\left\{
                                                                                                 \begin{array}{ll}
                                                                                                   \frac{4}{(1-|x|^2)^2}, & \hbox{for $i=j$ ;} \\
                                                                                                   0, & \hbox{for $i\ne j$.}
                                                                                                 \end{array}
                                                                                               \right.\end{equation}
A mapping $u\in C^{2}(\mathbb{B}^{n}, \mathbb{C})$ or more generally $u\in C^{2}(\mathbb{B}^{n}, \mathbb{R}^k)$ is said to be {\it hyperbolic harmonic} if $u$ (see Rudin \cite{rudin} and Stoll \cite{stoll} ) satisfies the hyperbolic Laplace equation

\begin{equation}\label{hyphar}
\Delta_{h}u(x)=(1-|x|^2)^2\Delta u(x)+2(n-2)(1-|x|^2)\sum_{i=1}^{n} x_{i} \frac{\partial u}{\partial x_{i}}(x)=0,
\end{equation}
 where
 $\Delta$ denotes the usual Laplacian in $\mathbb{R}^{n}$. We call $\Delta_{h}$ the {\it hyperbolic Laplacian operator}. This equation can be derived from equation \eqref{cron} above, if we consider $u$ as a mapping between the hyperbolic ball $(\mathbb{B}, (h_{ij})_{i,j})$ and the Euclidean space $\mathbb{C}$ or $\mathbb{R}^k$ with the flat metric  \begin{equation}\label{em}g_{ij}(x)=\left\{
                                                                                                 \begin{array}{ll}
                                                                                                 1, & \hbox{for $i=j$ ;} \\
                                                                                                   0, & \hbox{for $i\ne j$.}
                                                                                                 \end{array}
                                                                                               \right.\end{equation}
For $x\in \mathbb{B}$ we define the area element of the unit ball by
\begin{equation}\label{eq:invariant-measure-hyperbolic}
	d\tau(x) = \frac{2^n}{(1-|x|^2)^n} \ dV(x).
	\end{equation}
We will also sometimes denote hyperbolic measure of the set $A$ by $|A|_h= \tau(A)$.
The Poisson kernel for $\Delta_h$ is defined by $$P_h(x,\zeta)=\frac{(1-|x|^2)^{n-1}}{|x-\zeta|^{2n-2}}, \ \ (x,\zeta)\in \B\times \mathbb{S}.$$
Then for fixed $\zeta$, $x\to P_h(x,\zeta)$ is $\mathcal{M}-$harmonic, and, given a map $f\in \mathcal{L}(\mathbb{S})$, the function $$u(x) = P_h[f](x):=\int_{\mathbb{S}}P_h(x,\zeta) f(\zeta) d\sigma(\zeta)$$ is the Poisson extension of $f:$ it is $\mathcal{M}-$harmonic in $\B$, and agrees with $f$ on $\mathbb{S}.$

 We say that a smooth, real function $u$ is $\mathcal{M}-$subharmonic if $\Delta_h u(x) \ge 0$\footnote{This definition can be extended to the case of upper semicontinuous functions, by using the so-called invariant mean value property \cite{stoll}}.
 Notice further that
 \begin{equation}\label{fol}(\Delta_h u)(m(x)) =\Delta_h (u\circ m)(x),
 \end{equation}
 for every M\"obius transformation $m\in\mathcal{M}$ of the unit ball onto itself. For this reason, we shall also call hyperbolic harmonic (resp. hyperbolic subharmonic) functions $\mathcal{M}-$harmonic (resp. $\mathcal{M}-$subharmonic) functions. Note that, for $n=2$, the $\mathcal{M}-$harmonic and $\mathcal{M}-$subharmonic functions coincide with the usual Euclidean harmonic and subharmonic functions.

If $f$ is $\mathcal{M}-$subharmonic, then we have  the following Riesz decomposition theorem of Stoll \cite[Theorem~9.1.3]{stoll}:  $$f(x) = F_f(x)- \int_{\B} G_h(x,y) d\mu_f(y),$$ provided that $f\in \mathcal{S}^1$, where  $F_f(x)$ is the least $\mathcal{M}-$harmonic majorant of $f$ and $\mu_f$ is the $\mathcal{M}-$ Riesz measure of $f$, and $G_h(x,y)$ is the Green function of $\Delta_h$. If $f\in \mathcal{S}^p$, where $p>1$, then $g(x)=F_f(x)=P_h[\hat f](x)$, where $\hat f$ is the boundary function of $f$ (\cite[Theorem~7.1.1]{stoll}).
It follows from the formula \eqref{fol}, by putting $u=\mathrm{Id}$, and $m\in\mathcal{M}$, that
$$\Delta_h m=2 (n-2)(1-|m|^2)m.$$
Thus, M\"obius transformations are (considered as vectorial functions) hyperbolic harmonic only in the case $n=2$.
 %By putting $u(x) = g(|x|)$ and inserting in \eqref{hyphar} we arrive to the equation $$\Delta_h u= \left(1-r^2\right) \left(\frac{\left(2 (-2+n) r^2+(n-1) \left(1-r^2\right)\right) g'(r)}{r}+\left(1-r^2\right) g''(r)\right), $$ where $r=|x|$.
 \subsection{Appropriate weights} In order to construct the weights used in our definition of Bergman spaces, we need to define a couple of preliminary concepts. Thus, we define the hypergeometric functions $F$, as satisfying
  $$F\left[\begin{array}{c}
                                           a,b,c \\
                                           u,v
                                         \end{array};t\right]:=\sum_{n=0}^{\infty}\frac{(a)_{n}(b)_{n}(c)_n}{n!(u)_n (v)_n}t^{n},\enspace \mbox{for}\enspace|t|<1,$$ and
                                         $$F\left[\begin{array}{c}
                                           a,b,c
                                         \end{array};t\right]:=\sum_{n=0}^{\infty}\frac{(a)_{n}(b)_{n}}{n!(c)_n }t^{n},\enspace \mbox{for}\enspace|t|<1,$$
   and by analytic continuation elsewhere. Here $(a)_n$ denotes { the} rising factorial, i.e., $(a)_{n}=a(a+1)...(a+n-1)$, where $a \in \R$ is arbitrary.

Then one solution to the equation $\Delta_h \log v= -4(n-1)^2$ is given by
$$v=\Phi_n(r)=\exp\left\{\frac{(n-1) (2-n) r^2}{ n} F\left[\begin{array}{c}
                                           1,1,2-\frac{n}{2} \\
                                           2,1+\frac{n}{2}
                                         \end{array}; r^2\right]\right\} \left(1-r^2\right)^{n-1}.$$
                                         We will sometimes write $\Phi_n(x)$ instead of $\Phi_n(|x|)$.
                                         Note that $\Phi_n(r)\le \left(1-r^2\right)^{n-1}$, with strict inequality for $n>2$ and $r>0$. If $n=2$, then $\Phi_n(|x|) = (1-|x|^2)$ and this coincides with the case treated in \cite{ramostilli}. Observe, moreover, that $E_n (1-r^2)^{n-1} < \Phi_n(r),$ where $E_n =\exp\left\{\frac{(n-1) (2-n)}{ n} F\left[\begin{array}{c}
                                         	1,1,2-\frac{n}{2} \\
                                         	2,1+\frac{n}{2}
                                         \end{array}; 1 \right]\right\}.$ This shows that the Bergman norm induces by $\Phi_n$ is equivalent to that induced by the weight $(1-r^2)^{n-1}.$

                                      Finally, note that, for $n=3,4$ we have explicit formulas for $\Phi_n$:
                                         $$\Phi_3(r)=e^2 \left(\frac{1-r}{1+r}\right)^{\frac{1+r^2}{r}},  \ \ \ \Phi_4(r)=e^{-\frac{3}{2} r^2} (1 - r^2)^3.$$

%Note that
%$$\Delta_h (1-|x|^2)^{n-1}=-2 (n-1) n \left(1-|x|^2\right)^n$$

%$$\|f\|^p_{p,\alpha}=\int_{\B} |f|^p \Phi_n^\alpha (1-|x|^2)^{-n}dx $$

\subsection{Admissible monoid}\label{ssec:admissible}
We define $\mathfrak{E}$ to be the set of real-analytic  functions $g$ in $\mathbb{B}$  so that
$f:=\log |g|$ is $\mathcal{M}-$subharmonic.  Let  $\mathfrak{E}_+=\{f\in \mathfrak{E}: f>0\}$. Observe that $\mathfrak{E}$ is a monoid where the operation is simply the multiplication of two functions.  Observe that $1=e^0$, so $1\in \mathfrak{E}$. This monoid also contains the Abelian group $\mathcal{G}=\{e^f: \Delta_hf=0\}$.
Then for $a,b\in\mathfrak{E}_+$, $c, d\in \mathfrak{E}$,  $p\ge 0$ and $\alpha, \beta>0$ we have
\begin{enumerate}
	\item $a\cdot b \in\mathfrak{E}_+$, $c\cdot d\in \mathfrak{E}$,
	\item $a^p\in \mathfrak{E}_+$, $c^p\in \mathfrak{E}$,
	\item $\alpha a +\beta b \in \mathfrak{E}_+$.
	\item $\exp(a)\in \mathfrak{E}_+, \ \ \exp(c)\in \mathfrak{E}_+$.
	\item\label{prop:subh} $f \in \mathfrak{E} \Rightarrow |f|$ is $\mathcal{M}-$subharmonic. 
\end{enumerate}

In other words, $\mathfrak{E}_+$ is a convex cone and at the same time a monoid. We refer the reader to the monograph by M. Stoll \cite{stoll} for more details (see also \cite{kalaj}). 

% \item the last $\mathcal{M}-$harmonic majorant of $f$, $F_f$ coincides with the $\mathcal{M}-$Poisson extension of $f|_{\mathbb{S}}$ to the unit ball $\mathcal{B}$.
% \item every $g$ can be approximated by a sequence of continuous functions $g_k$ defined in $\overline{\mathbb{B}}$ uniformly in $\mathbb{B}$ and in the norm  $\|\cdot \|_{{B}^p_\alpha}^p$ so that $g_k$ satisfies (1) and (2) for every $k$.
 %     \end{enumerate}
 %Note that the condition (2) and (3) of the previous definitions are redundant, provided that $n=2$.

 %In that case we can define $g_n(z) = g(\frac{n}{n+1} z)$. This construction is not suitable for $n>2$, because the dilatation of a $\mathcal{M}-$harmonic (or a subharmonic) function is not  $\mathcal{M}-$harmonic (or a subharmonic) in general.

\begin{definition}
For $0 < p < \infty$ and $\alpha > 1$ we say that a smooth function $f$  in $\mathbb{B}$ belongs to the $\mathcal{M}-$Bergman space ${B}^p_\alpha$ if
$$\|f\|_{{B}^p_\alpha}^p = c(\alpha)\int_\mathbb{B} |f(x)|^p \Phi_n^\alpha(|x|) d\tau(x) < \infty,$$ where
$$\frac{1}{c(\alpha)}=\left( \int_{\B} \Phi^\alpha_n(x)(1-|x|^2)^{-n}\frac{dV(x)}{{\omega_n}}\right).$$
\end{definition}
\begin{definition}
 For $0 < p < \infty$ and $\alpha > 1$ we define
the $\mathcal{M}-$Bergman monoid $\mathbf{B}^p_\alpha$ consisting of functions $f$ in ${B}^p_\alpha\cap \mathfrak{E}$ having harmonic majorant $F_f\in {B}^p_\alpha$. % satisfying the additional condition that $|f(x)|^p \Phi_n(x)^\alpha\to 0$ uniformly when $|x|\to 1$.
\end{definition}
%\begin{remark}
% Note that, those "additional conditions" are redundant for bounded functions. They  are also redundant for $n=2$ in the weigh-Bergman and Hardy space of Holomorphic functions as it was observed by Kulikov in \cite{kulik}.  The proofs are given in Proposition~\ref{propo1} and Proposition~\ref{propozicija} below. On the other hand, "additional conditions" are redundant for $n=2$ and log-subharmonic functions as well. Namely, if $f$ is log-subharmonic, then it coincides on the boundary with the  modulus of a holomorphic function $h$, almost everywhere. This implies that those "additional conditions" are also redundant in this case, at least for $p>1$ (Proposition~\ref{propo1} and Proposition~\ref{propozicija}).  We will not use this claim in the proof of the Corollary~\ref{merkur}, but we will use the approximation method by using the so-called dilatations, which are suitable for the plane, but not for the space.
%\end{remark}
Observe that for $f(x) \equiv 1$ we have $ \|f\|_{{\mathbf{B}^q_\alpha}} = 1$ for all  $p>0$ and $\alpha>1$. Henceforth we will use the notation $\|f\|_{\mathbf{B}^p_{\alpha}}$ interchangeably with $\|f\|_{B^p_{\alpha}}$ whenever $f \in \mathbf{B}^p_{\alpha}.$ 

\subsection{M\"{o}bius invariance of $\mathbf{B}^2_\alpha$ }
It is important to mention that the M\"{o}bius group acts not only on the measure $\tau$ but on the spaces $\mathbf{B}^2_\alpha$ as well. More precisely, given a function $f\in \mathbf{B}^2_\alpha$ and $m\in \mathcal{M}$, the function
\begin{equation}\label{prety}g(x) = f\left(m(x)\right)\frac{\Phi_n^{\alpha/2}(|m(x)|)}{\Phi_n^{\alpha/2}(|x|)}\end{equation}
also belongs to the space $\mathbf{B}^2_\alpha$ and moreover it has the same norm as $f$  with respect to the measure $\tau$. We thus only need to check that $\Delta_h \log |g|\ge 0$, if $\Delta_h \log |f|\ge 0$, and this follows from the formula \eqref{fol} and straightforward calculations:
\[\begin{split}\Delta_h\log |g(x)|&=\Delta_h \log(|f\left(m(x)\right)|)+\Delta_h \log {\Phi_n^{\frac{\alpha}{2}}(|m(x)|)} - \Delta_h \log {\Phi_n^{\frac{\alpha}{2}}(|x|)}
	\\&\ge 0 + (4(n-1)^2-4(n-1)^2)\frac{\alpha}{2}= 0.\end{split}\]

Finally, we also point out the well-known formula for the Jacobian of M\"obius transformations of the unit ball onto itself $$J_m(x)=\frac{(1-|m(x)|^2)^n}{(1-|x|^2)^n}.$$ See e.g. \cite[p.~vii]{stoll}.

%Then
%$$g(s) = s \exp\left[-\int_{\mu(s)}^{+\infty}{\gamma}{\Upsilon(x) dx}\right],$$ where $\gamma=2^n \beta (n-1)^2$, is decreasing and constant if $v\equiv 1$.

%Here $$\mu(t)=|\{x: v^a(x) \Phi^\beta_n(x)\ge t\}|_h,$$ $a>0,\beta>0$  and $v$ is $\mathcal{M}-$log-subharmonic function and
%$$\Delta_h \log \Phi_n(x) = (n-1)^2.$$

\section{Main theorem}

With all the relevant notions already having been introduced, our goal is to maximize  the following functional:
\begin{equation}\label{newco}
R_n(f,\Omega)=
\frac{c(\alpha) \int_\Omega  |f(x)|^2 \Phi_n^\alpha(|x|) \, d\tau(x)}{\|f\|^2_{\mathbf{B}^2_\alpha}}
\end{equation}
over all $f \in\mathbf{B}^2_\alpha$ and $\Omega \subset \mathbb{B}$ with $\tau(\Omega) = s.$ The main result of this paper extends the corresponding result \cite{ramostilli} to the higherdimensional space. We prove

\begin{theorem}\label{thm:main} Let $\alpha>1,$ and $s>0$ be fixed. Among all functions $f\in \mathbf{B}^2_\alpha$ and among
all measurable sets $\Omega\subset  \mathbb{B}$ such that $\tau(\Omega)=s$, the quotient $R_n(f,\Omega)$ as defined in \eqref{newco} satisfies the inequality
\begin{equation}\label{upper}
	R_n(f,\Omega) \le R_n(1, \mathbb{B}_s),
\end{equation}
where $ \mathbb{B}_s$ is a ball centered at the origin with $\tau( \mathbb{B}_s) = s.$ Moreover, there is equality in \eqref{upper} if and only if $f$ belongs to the extremal set  and $\Omega$ is a ball such that $\tau(\Omega)=s$. The extremal set consists of functions  $g(x) = \frac{\Phi_n^{\alpha/2}(|m(x)|)}{\Phi_n^{\alpha/2}(|x|)}$, where $m\in\mathcal{M}$ is a M\"obius transformation of the unit ball onto itself.
\end{theorem}

It is worth noting that A. Kulikov, F. Nicola, J. Ortega-Cerd\`a and P. Tilli \cite{KNOT} have recently independently obtained a result (see Theorem 5.1 in their manuscript) which, when specified to the $n-$dimensional hyperbolic space endowed with the Laplace-Beltrami operator induced by the hyperbolic metric, gives an upper bound equivalent to that of Theorem \ref{thm:main} above. We opted, however, to keep a structure for the manuscript in which the proof of the main result highlights the intrinsic hyperbolic geometry of the problem, since our current proof yields, among others, an explicit form for the extremizers.

In order to prove Theorem~\ref{thm:main}, we need some ingredients present both in \cite{3} as well as in \cite{ramostilli}.  The first such ingredient is the isoperimetric inequality.

For a Borel set $\Omega\subset \mathbb{B}$ we recall the definition \eqref{eq:invariant-measure-hyperbolic} of hyperbolic volume   $$\tau(\Omega) =\int_{\Omega}\left(\frac{2}{1-|x|^2}\right)^n dx,$$ and define  the hyperbolic perimeter by
$$P(\Omega)=\int_{\partial \Omega}\left(\frac{2}{1-|x|^2}\right)^{n-1} d\mathcal{H}^{n-1}(x),$$ where $\mathcal{H}^{n-1}$ is the $n-1$ dimensional Hausdorff measure.
Assume that  $\IB(r)$ is a ball centered at zero and with Euclidean radius $r$. Then its hyperbolic volume is $$V_r=\tau(\mathbb{B}(r))=\frac{2^n \pi ^{n/2} F\left[\frac{n}{2},n,\frac{2+n}{2},r^2\right] r^n}{ \Gamma\left[1+\frac{n}{2}\right]},$$ and its hyperbolic perimeter is $$P_r=P(\mathbb{B}(r))=\frac{2 n \pi ^{n/2} r^{n-1} \left(1-r^2\right)^{1-n}}{\Gamma\left[1+\frac{n}{2}\right]}.$$

The isoperimetric property of hyperbolic ball was established by E. Schmidt \cite{14} see also \cite{cvpde, 15}. He proved that for every Borel set $\Omega\subset \mathbb{B}$ of finite perimeter $P(\Omega)$, so that  $\tau(\Omega)=\tau(\mathbb{B}(r))$ and $r>0$ the following inequality holds:
\begin{equation}\label{isophyper}
P_r\le P(\Omega),
\end{equation}
where $P_v$ is the perimeter of $\IB(v)$. Since $r\to v=V_r$ is increasing, we may define its inverse function $S(v)=r$.
Then we define the function $\Upsilon$ by  \begin{equation}\label{newper}\Upsilon(v)=\frac{v}{P^2_{S(v)}} \end{equation} and thus, \eqref{isophyper} may be written as
\begin{equation}\label{newper1}\frac{P(\Omega)^2}{\tau(\Omega)} \ge \frac{1}{\Upsilon(\tau(\Omega))},\end{equation} with equality in \eqref{newper1} if and only if $\Omega$ is a ball.

\begin{remark}
If $n=2$ then  $\Upsilon(V)=\frac{1}{4\pi + V}$. It seems unlikely that we can give an explicit expression for the function $\Upsilon$ for the higher-dimensional case, but, as we shall see, this does not interfere in our proof.
\end{remark}

Another fact we will need for the proof, also common in \cite{3} and \cite{ramostilli}, is that point evaluation functionals are also continuous in this case, as highlighted by the following proposition:

\begin{proposition}\label{propo1}
Let $p>1$ and $\alpha>1$. Then point evaluations are continuous functionals in $\mathbf{B}_\alpha ^p$.
More precisely there is a constant  $C=C(n,\alpha)$ so that, if $|f|$ belongs to the Bergman space $\mathbf{B}_\alpha^p$, then
\begin{equation}\label{seconda}|f(x)|^{p}\Phi_n^\alpha(x)\le C\|f\|^p_{\mathbf{B}_\alpha^p}.\end{equation}
Moreover, we also have $\lim_{|x|\to 1}|f(x)|^p(1-|x|^2)^\alpha=0$.
\end{proposition}

For a proof, we refer the reader to \cite{kalaj}. We are now ready to prove our main result. 

\begin{proof}[Proof of Theorem \ref{thm:main}]
	
	Let $u(x) = |f(x)|^2 \Phi_n^\alpha(|x|)$ and assume that $$\|f\|^2_{\mathbf{B}^2_\alpha}=\int_{\mathbb{B}}|f(x)|^2 \Phi_n^\alpha(|x|)d\tau(x)=1.$$ Let $u^*(t)$ be defined as
\[
u^*(t) := \inf_{s > 0} \{s \ge 0 \colon \mu(s) \le t\},
\]
where $\mu(s) = \tau(\{ u > s\}).$ Let now $\Omega_s=\{x: u(x)>u^\ast(s)\}$. Then by Proposition~\ref{propo1}, $\Omega_t$ is strictly contained in $\mathbb{B}$. Now, we can already start with the necessary modifications to \cite{ramostilli}.

Indeed, for $n\ge 2$, define

$$I_n(s) =  \int_{\{x: u(x)> u^\ast(s)\}} u(x) d\tau(x).$$ Observe that $u^\ast(0)=\infty$ and $u^\ast(\infty)=0$. Thus $I_n(0)=0$ and $I_n(\infty)=\|f\|^2_{\mathbf{B}^2_\alpha}=1$. Notice also that
\[
\int_{\Omega} u(x) d\tau(x) \le I_n(s),
\]
whenever $\tau(\Omega) = s.$

%For exposition let us first consider the case $n=2$. Let  $u^\ast(s)$ be the unique value of $t$ so that $\tau\{x: u(x)>t\}=s$.

%Tilli and the second author proved that if $u(x) = |f(x)|^2 (1-|x|^2)^{\alpha}$, and $d\tau(x)=4(1-|x|^2)^{-2}dxdy$, then for $$I(s) = \int_{\{x: u(x)>u^\ast(s)\}} u(x)d\tau(x)$$ we have  \begin{equation}\label{Ipp}I''(s)  +\alpha I'(s)(4\pi+s)^{-1}\ge 0.\end{equation} Further  for $$\theta(s) = 1-(1+s/(4\pi))^{1-\alpha}$$ and $$T(t) = \theta^{-1}(t)=4\pi  \left(-1+(1-t)^{\frac{1}{1-\alpha}}\right)$$ we have $$J(t)=I(T(t))$$ is convex. In other words $$J''(t)\ge 0.$$

%Since $J(0)=0$ and $J(1)= 1$ we have $J(t)\le t$. So $$I(s)\le \theta(s).$$

We now state a lemma which allows us to explicitly compute derivatives of the distribution function $\mu,$ as well as that of the hyperbolic rearrangement $u^*.$ We highlight the fact that it can be proved in a similar way as in \cite[Lemma~3.2]{3}.

\begin{lemma}\label{thm:lemma-derivatives} The function $\mu$ is absolutely continuous on $(0,\max u],$ and
	\[
	-\mu'(t) = 2^n\int_{\{u = t\}} |\nabla u|^{-1} (1-|z|^2)^{-n} \, d \mathcal{H}^{n-1}.
	\]
	In particular, the function $u^*$ is, as the inverse of $\mu,$ locally absolutely continuous on $[0,+\infty),$ with
	\[
	-(u^*)'(s) = \left( 2^n\int_{\{u=u^*(s)\}} |\nabla u|^{-1} (1-|z|^2)^{-n} \, d \mathcal{H}^{n-1} \right)^{-1}.
	\]

\end{lemma}

Let us then denote the boundary of the superlevel set where $u > u^*(s)$ as
\[
A_s=\partial\{u>u^*(s)\}.
\]
We have then, by Lemma \ref{thm:lemma-derivatives},
\[
I_n'(s)=u^*(s),\quad
I_n''(s)=-2^{-n}\left(\int_{A_s} |\nabla u|^{-1}(1-|x|^2)^{-n}\,d{\mathcal H}^{n-1}(x)\right)^{-1}.
\]

The Cauchy-Schwarz inequality implies
\[
\begin{split}&\left(\int_{A_s} 2^n|\nabla u|^{-1}(1-|x|^2)^{-n}\,d{\mathcal H}^{n-1}(x)\right)
\left(\int_{A_s} 2^{n-2}|\nabla u|(1-|x|^2)^{2-n}\,d{\mathcal H}^{n-1}\right)
\\&\geq
\left(\int_{A_s} 2^{n-1}(1-|x|^2)^{1-n}\,d{\mathcal H}^{n-1}(x)\right)^2.\end{split}
\]
Letting
\[
L(A_s):= 2^{n-1}\int_{A_s} (1-|x|^2)^{1-n}\,d{\mathcal H}^{n-1}
\]
denote the hyperbolic length of $A_s$, we obtain the lower bound
\begin{equation}\label{eq:lower-bound-second-derivative}
I_n''(s)\geq - 2^{n-2}\left(\int_{A_s}|\nabla u|(1-|x|^2)^{2-n} d{\mathcal H}^{n-1}\right)
L(A_s)^{-2}.
\end{equation}

Let $\nu=\nu(x)$ be the outward unit normal to $A_s$ at a point $x$. Note that $\nabla u$ is parallel to $\nu$ but directed in the opposite direction. Thus we have $|\nabla u| = -\left<\nabla u,\nu\right>$. Also, we note that since for $x\in A_s$ we have $u(x) = u^\ast(s)=t$, we obtain for $x\in A_s$ that
$$\frac{|\nabla u(x)|}{t} = \frac{|\nabla u(x)|}{u} =  \left<\nabla  \log u(x), \nu\right>.$$

Now the integral in \eqref{eq:lower-bound-second-derivative} can be evaluated by the Gauss divergence theorem:
\[\begin{split}\int_{A_s} \frac{|\nabla u| dH^{n-1}}{(1-|x|^2)^{n-2}}&=-t  \int_{\Omega_s}\mathrm{div}\left(\frac{\nabla \log u(x)}{(1-|x|^2)^{n-2}}\right) dx
\\&=-t  \int_{\Omega_t}\frac{1}{(1-|x|^2)^n}\Delta_h {\log u(x)} dx.\end{split}\]

Now we plug
$u= g(x)^2 \Phi^\alpha_n(x)$, where $g(x)=|f(x)|$, and calculate $$-t \Delta_h \log(g^2 \Phi^\alpha_n)=-(2 t\Delta_h \log g+t\alpha \Delta_h \log \Phi_n)\le 0+4t \gamma, $$ where
$\gamma=\alpha (n-1)^2.$ Thus $$\int_{A_s} \frac{|\nabla u| dH^{n-1}}{(1-|x|^2)^{n-2}}\le \frac{4t\gamma \tau(\Omega_t)}{2^n}=\frac{4\gamma t s}{2^n}=\frac{4\gamma s u^\ast(s)}{2^n}=2^{2-n}\gamma s I_n'(s).$$

By using \eqref{eq:lower-bound-second-derivative} and previous equation we obtain

\begin{equation}\label{fin}I_n''(s)\geq -\gamma s I_n'(s) L(A_s)^{-2}.
\end{equation}

Let $L(A_s)=P(\Omega_s)$. By \eqref{newper1}, we have
\begin{equation}\label{ypsibaraz} -L(A_s)^{-2} \ge -\frac{\Upsilon(s)}{s}\end{equation}
with  equality in \eqref{ypsibaraz} if and only if $v$ is a constant, because in that case $\Omega_s$ is a ball centered at the origin.
Then \eqref{fin} implies
\begin{equation}\label{fin1}I_n''(s)\geq -\gamma  I_n'(s) {\Upsilon(s)}.\end{equation}
%Observe that for  $n=2$, $$I''(s)\geq -\alpha  I'(s) \frac{1}{4\pi  +s},$$ which coincides with \eqref{Ipp}.
Let  $J(x) = I_n(T(x))$, where $T(x) = \theta^{-1}(x)$, and $\theta (s) = R_n(1,\mathbb{B}_s)$. Then we have $$J'(x)=I_n'(T(x))T'(x), \ \  J''(x) = I_n''(T(x))(T'(x))^2+I_n'(x) T''(x).$$ We now claim that $$J''(x)\ge 0.$$ In view of \eqref{fin1}, we need to show that $$\frac{T''(x)}{\left(T'(x)\right)^2}=\gamma \Upsilon (T(x))$$
where
$\gamma=\alpha(n-1)^2$. The last equation is equivalent to

\begin{equation}\label{needed}\frac{\theta''(s)}{\theta'(s)}=-\gamma \Upsilon (s).\end{equation}

Here,

$$\theta(s)=n c_\alpha \int_{0}^{v(s)}s^{n-1}\Phi_n^\alpha(s)(1-s^2)^{-n} ds$$ and $v(s)$ is the Euclidean radius of the hyperbolic ball $\mathbb{B}_s$ of area $s$.
Observe that $\theta(0)=0$ and, because $v(\infty)= 1$, we obtain that $$\theta(\infty)=n c_\alpha \int_{0}^{1}s^{n-1}\Phi_n^\alpha(s)(1-s^2)^{-n} ds=1.$$
Then $T(0)=0$ and $T(1)=\infty$, and so $J(0)=I_n(0)=0$ and $J(1)=I_n(\infty)=1$. We therefore recall the definition of the weight function $\Phi_n:$
 $$\Phi_n(r)=\exp\left\{\frac{(n-1) (2-n) r^2}{ n} F\left[\begin{array}{c}
                                           1,1,2-\frac{n}{2} \\
                                           2,1+\frac{n}{2}
                                     \end{array}; r^2\right]\right\} \left(1-r^2\right)^{n-1}.$$
Then we have
\begin{equation}\label{beq}\begin{split}\frac{\theta''(s)}{\theta'(s)}&=-\frac{\left(-1+n+(1+n) v(s)^2\right) v'(s)}{v(s) \left(-1+v(s)^2\right)}+\frac{v''(s)}{v'(s)}+ \alpha\frac{\Phi_n'(v(s)) v'(s)}{\Phi_n(v(s))}.\end{split}\end{equation}
We need to gather more information on $v(s).$ Note that

\begin{equation}\label{eq:volume-property}
	\tau(\mathbb{B}(v(s)))=\frac{2^n \pi ^{n/2} F\left[\frac{n}{2},n,\frac{2+n}{2},v(s)^2\right] v(s)^n}{ \Gamma\left[1+\frac{n}{2}\right]}=s,
\end{equation}

and, moreover,

 $$P(\mathbb{B}(v(s)))=\frac{2 n \pi ^{n/2} v(s)^{-1+n} \left(1-v(s)^2\right)^{1-n}}{\Gamma\left[1+\frac{n}{2}\right]}.$$

Since $\Upsilon =\frac{\tau(\mathbb{B}(v(s)))}{P^2(\mathbb{B}(v(s)))}$ by definition, we have
$$\Upsilon(s)=\frac{ \left(1-v(s)^2\right)^{-2+2 n} \Gamma\left[\frac{n}{2}\right] F\left[\frac{n}{2},n,\frac{2+n}{2},v(s)^2\right]}{n 2^{n-1} \pi ^{n/2} v(s)^{n-2}}.$$
By differentiating \eqref{eq:volume-property} and using that the derivative of the function $s \mapsto F[\frac{n}{2}, n, \frac{n+2}{2},s]$ equals
\[
\frac{n((1-s)^{-n} - F[\frac{n}{2}, n, \frac{n+2}{2},s])}{2s},
\]
we obtain that
$$\frac{2^n n \pi ^{n/2} v(s)^{-1+n} \left(1-v(s)^2\right)^{-n} v'(s)}{\Gamma\left[1+\frac{n}{2}\right]}=1, \text{ and hence }$$
\begin{equation}\label{eq:derivative-radius}
	v'(s)=\frac{2^{-n} \pi ^{-n/2} v(s)^{1-n} \left(1-v(s)^2\right)^n \Gamma \left[1+\frac{n}{2}\right]}{n}.
\end{equation}
Differentiating once more then yields
 \begin{equation}\label{claim}-\frac{\left(-1+n+(1+n) v(s)^2\right) v'(s)}{v(s) \left(-1+v(s)^2\right)}+\frac{v''(s)}{v'(s)}=0.\end{equation}
Finally, from the Euler transformation formula
\begin{equation} \label{eq:Euler-formula}
F[a,b,c,x]=(1-x)^{c-a-b}F[c-a, c-b, c,x],
\end{equation}
 applied to $a=1$, $b=2-n/2$, $c=1+n/2$, and the definition of $\Phi_n,$ we have  \[\begin{split}&\frac{\Phi_n'(v(s)) v'(s)}{\Phi_n(v(s))}
\\&=(n-1) v(s) \left(\left(-2+\frac{4}{n}\right) F\left[1,2-\frac{n}{2},\frac{2+n}{2},v(s)^2\right]-\frac{2}{1-v(s)^2}\right) v'(x)
\\ \text{ (by  \eqref{eq:derivative-radius})}  &=\frac{ (n-1) \Gamma\left[1+\frac{n}{2}\right]   \left(1-v(s)^2\right)^n \left(\left(\frac{4}{n}-2\right) F\left[1,2-\frac{n}{2},\frac{2+n}{2},v(s)^2\right]-\frac{2}{1-v(s)^2}\right)}{n2^{n} \pi ^{n/2} v(s)^{n-2}}
%\\&=\frac{\alpha (n-1) \Gamma\left[1+\frac{n}{2}\right]    \left(-2 \left(1-v(s)^2\right)^{n-1}+\left(-2+\frac{4}{n}\right) \left(1-v(s)^2\right)^{-2+2n} F\left[-1+n,\frac{n}{2},1+\frac{n}{2},v(s)^2\right]\right)}{n2^{n} \pi ^{n/2} v(s)^{n-2}}
\\  \text{ (by  \eqref{eq:Euler-formula})} &=\frac{(n-1) \Gamma\left[1+\frac{n}{2}\right]    \left(\left(\frac{2}{n}-1\right) F\left[n-1,\frac{n}{2},1+\frac{n}{2},v(s)^2\right]- \left(1-v(s)^2\right)^{1-n}\right)}{n2^{n-1} \pi ^{n/2}  \left(1-v(s)^2\right)^{2-2n}v(s)^{n-2}}
\\&=\frac{ (n-1) \Gamma\left[1+\frac{n}{2}\right]   \left(1-v(s)^2\right)^{2n-2}  \left(\left(\frac{2}{n}-2\right) F\left[n,\frac{n}{2},1+\frac{n}{2},v(s)^2\right]\right)}{n2^{n-1} \pi ^{n/2} v(s)^{n-2}}.
\end{split}\]
The last equation follows from the identity $$\frac{2 \Gamma[m+n-1]}{(2 m+n) \Gamma[1+m] \Gamma[n-2]}+\frac{2 \Gamma[m+n-1]}{\Gamma[1+m] \Gamma[n-1]}=\frac{4 \Gamma[m+n]}{(2 m+n) \Gamma[1+m] \Gamma[n-1]},$$
and then expanding the Taylor series of $\left(\frac{2}{n}-1\right) F\left[n-1,\frac{n}{2},1+\frac{n}{2},r\right]- \left(1-r\right)^{1-n}.$ Therefore

\begin{equation}\label{merk}
\frac{\Phi_n'(v(s)) v'(s)}{\Phi_n(v(s))}=\frac{(n-1)^2 \Gamma\left[\frac{n}{2}\right]      F\left[n,\frac{n}{2},1+\frac{n}{2},v(s)^2\right]}{n2^{n-1} \pi ^{n/2} \left(1-v(s)^2\right)^{2-2n}v(s)^{n-2}}.
\end{equation}

Now \eqref{beq}, \eqref{claim}, \eqref{merk} imply \eqref{needed}. Finally, since $J(0)=0$ and $J(1)\le 1,$ we have $J(t)\le t$. So \begin{equation}\label{eq:final-ineq}
	 I_n(s)\le \theta(s).
\end{equation}
This proves \eqref{upper} in Theorem \ref{thm:main}. Finally, in order to characterize the equality case, note that, if there is equality at \eqref{eq:final-ineq} at one point $s_0 \in (0,1),$ then it must hold for \emph{all} points. But then \eqref{fin1} becomes an \emph{equality} for all $s \in (0,1),$ and as a consequence, we must have equality in the hyperbolic isoperimetric inequality \eqref{ypsibaraz}. Thus, the set $A_s$ must be the boundary of an euclidean ball.

Moreover, $\Omega = \{ u > u^*(s) \}$ up to a set of measure zero. We then apply a suitable M\"obius transformation $m$ which takes $\{ u > u^*(s) \}$ -- which we know is a ball -- into a ball of same hyperbolic measure centered at the origin. Consider the associated function $g(x) = f(m(x)) \frac{\Phi_n^{\alpha/2}(|m(x)|)}{\Phi_n^{\alpha/2}(|x|)},$ and $v(x) = |g(x)|^2 \Phi_n^{\alpha}(|x|).$ By construction, the set $m^{-1}(\Omega) = \{ v > v^*(s)\}.$ Thus, both $v$ and $\Phi_n$ are constant on the boundary of $m^{-1}(\Omega),$ which implies that $\log |g(x)|$ is  also constant on $\partial m^{-1}(\Omega).$

Finally, in order for \eqref{fin1} to be an equality, we must have $\Delta_h \log |f(x)| = 0$ in $\Omega,$ which implies by M\"obius invariance that $\Delta_h \log |g(x)| = 0$ almost everywhere in $m^{-1}(\Omega).$ As $\log |g(x)|$ is also constant on the boundary of that set, by the maximum principle for $\mathcal{M}-$subharmonic functions we have that $\log |g(x)|$ must be constant in $m^{-1}(\Omega).$ Since this argument can be used for arbitrary $s \in (0,1),$ $f$ must belong to the asserted extremal set in Theorem \ref{thm:main}, and the proof is complete.
 \end{proof}

%$$F\left[1,2-\frac{n}{2},\frac{2+n}{2},v(s)^2\right]=\left(1-p^2\right)^{-2+n} F\left[-1+n,\frac{n}{2},1+\frac{n}{2},p^2\right]$$

%$$X=-\alpha\left(\frac{2^{1-n} \pi ^{-n/2} p^{2-n} \left(1-p^2\right)^{-2+2 n} \Gamma\left[\frac{n}{2}\right] F\left[\frac{n}{2},n,\frac{2+n}{2},p^2\right]}{n} (n-1)^2\right)$$

%$$X=\frac{\alpha (n-1) \Gamma\left[1+\frac{n}{2}\right] \left(-\left(1-p^2\right)^{n-1} +\left(-1+\frac{2}{n}\right)\text{  }F\left[-1+n,\frac{n}{2},1+\frac{n}{2},p^2\right]\right) }{n2^{n-1} p^{n-2} \pi ^{n/2}}.$$

\section{Wavelet transforms and the main Theorem}

An important question, which arises naturally when comparing our Theorem \ref{thm:main} with the main results in \cite{ramostilli}, is whether our results have any relationship to the case of Wavelet transforms in higher dimensions.

More specifically, one may wonder whether the wavelet transform associated with a certain \emph{radial} window $\psi(x) = \psi(|x|),$ given explicitly by
\[
W_{\psi}f(y,t) = C_{\psi} t^{-n/2} \int_{\R^n} f(x) \overline{ \psi\left( \frac{x-y}{t} \right)} \, dx
\]
satisfies that, for some specific weight function $w(t),$ we could have $w(t) W_{\psi}f(-x,t) =: G_f(x,t)$ belong to the admissible monoid as defined in subsection \ref{ssec:admissible}. If this is indeed the case, then, by property \ref{prop:subh} above of the admissible monoid, $|G_f|$ would be $\mathcal{M}-$subharmonic. As subharmonicity is a broad property, one may search for a  slightly more rigid property to impose on $G_f.$ For instance, if $n=1,$ then the paper by the second author and P. Tilli \cite{ramostilli} uses a class of windows such that, for a suitable weight $w,$ the function $G_f$ above is \emph{analytic} in the upper-half plane. 

Since analyticity is a perhaps too strict property, a good intermediate goal seems to be to find all $w, \psi$ such that the functions $G_f$ generated as before are all  $\mathcal{M}-$harmonic in $\mathbb{H},$ the upper-half space model for the hyperbolic space. Thus, summarizing, we are interested in searching for all radial functions $\psi$ and all positive weights $w(t) > 0$ such that $w(t) W_{\psi}f(-x,t)$ is $\mathcal{M}-$harmonic in the upper half-space $\R^{n+1}_+.$ This is the content of the next result, whose proof is inspired by the explicit one-dimensional characterization of analytic wavelets by G. Ascensi and J. Bruna in \cite{AscensiBruna} (see also \cite{Holigetal}).

\begin{proposition}\label{prop:charact-harmonic}
Let $n > 1$ be fixed. Suppose that the admissible radial window $\psi \in L^2(\R^n)$  and the smooth weight $w:\R^n \to \R_+$ are such that, for each $f \in L^2(\R^n),$ the function $(x,t) \mapsto w(t) W_{\psi}f(-x,t)$ is $\mathcal{M}-$harmonic in $\R^{n+1}_+.$ Then $w(t) = t^{-\beta}$ and $\widehat{\overline{\psi}}(r) = r^{\beta} K_{n/2}(r),$ where $\beta>\frac{n}{2}$ and $K_{\nu}$ is the modified Bessel function of the second kind of order $\nu.$
\end{proposition}

\begin{proof} By the definition of the Wavelet transform and Plancherel's theorem, we may write
\begin{align*}
	w(t) W_{\psi}f(-y,t) & = C_{\psi} w(t) t^{n/2}  \int_{\R^n} \widehat{f}(\xi) e^{-i y\cdot \xi} \widehat{\overline{\psi}} (t\xi) \, d \xi \cr
	& =: C_{\psi} \eta(t) \int_{\R^n} \widehat{f}(\xi) e^{-i y\cdot \xi} \widehat{\overline{\psi}} (t\xi) \, d \xi.
\end{align*}
If this function is $\mathcal{M}-$harmonic for all $f \in L^{2}(\R^n)$, then we must have $\Delta_h (w(t) W_{\psi}f(-x,t)) = 0$ for each $x \in \R^n, t > 0.$ Thus, by density,  we must have the equality as distributions
\[
\Delta_h(\eta(t) e^{-i y \cdot \xi} \varphi(t|\xi|)) = 0,
\]
where we used the shorthand $\varphi = \widehat{\overline{\psi}}.$ On the other hand, since $w$ is smooth, elliptic regularity shows that $\varphi$ must be of class $C^2(\R_+),$ and hence this expression may be taken in the pointwise sense and then further rewritten, taking into consideration the definition of $\Delta_h F(y,t) = t^2 \Delta_y F(y,t) + t^2 \partial_t^2 F(y,t) - (n-1) t \partial_t F(y,t)$ in the upper-half space model. We obtain that
\begin{align*}
	e^{-i y \cdot \xi}  \{ \eta(t) \left( (t|\xi|)^2 \varphi(t|\xi|)  - (t|\xi|)^2 \varphi''(t|\xi|) + (n-1) \varphi'(t|\xi|) (t|\xi|) \right) + \cr
	t\eta'(t) \left( (n-1)\varphi(t|\xi|)  - 2 \varphi'(t|\xi|) (t|\xi|) \right) - t^2 \eta''(t) \varphi(t|\xi|) \} = 0
\end{align*}
holds whenever $t > 0, \xi \in \R^n.$ Relabeling $t|\xi| = r,$ and subsequently $e^u = t,$ with $\gamma(u) = \eta(e^u),$ and assembling similar terms, we obtain
\begin{align*}
	\gamma(u) \left( r^2 (\varphi(r) - \varphi''(r)) + (n-1) \varphi'(r) r \right) + \cr
	\gamma'(u)(n\varphi(r) - 2 \varphi'(r) \cdot r) - \gamma''(u) \varphi(r) = 0, \, \forall u \in \R, \, r > 0.
\end{align*}
Thus, the vectors $(\gamma(u), \gamma'(u), - \gamma''(u))$ and $( r^2 (\varphi(r) - \varphi''(r)) + (n-1) \varphi'(r) r, n\varphi(r) - 2 \varphi'(r) \cdot r, \varphi(r))$ are \emph{always} orthogonal. This implies one of the following:
\begin{enumerate}[(i)]
	\item $\gamma \equiv 0,$ which corresponds to a trivial solution $w \equiv 0;$
	\item $\varphi\equiv 0,$ which again corresponds to a trivial case;
	\item The dimension of the subspace spanned by  $( r^2 (\varphi(r) - \varphi''(r)) + (n-1) \varphi'(r) r, n\varphi(r) - 2 \varphi'(r) \cdot r, \varphi(r)), r > 0,$ is at most one, which implies that there is a constant $c > 0$ such that $\varphi'(r) \cdot r = c \cdot \varphi(r) \iff (\log \varphi)'(r) = \frac{c}{r}, \, r > 0 \iff \varphi(r) = C \cdot r^{c},$ for some $C \in \R.$ As this does \emph{not} yield an admissible wavelet, we may rule out this case;
	\item The dimension of the subspace spanned by $(\gamma(u), \gamma'(u), - \gamma''(u)), \, u \in \R,$ is at most one. This implies that $\gamma(u) = C \cdot e^{\alpha u},$ for some $C, \alpha \in \R.$ In that case, we obtain the following equation for $\varphi:$
	\begin{equation}\label{eq:differential}
		(n\alpha - \alpha^2 + r^2) \varphi(r) + (n-1 - 2\alpha) \varphi'(r) \cdot r - \varphi''(r) \cdot r^2 = 0, \, \forall \, r > 0.
	\end{equation}
	This is a second order linear equation in $\varphi.$ Moreover, it is not difficult to identify two particular solutions to \eqref{eq:differential}: $\varphi_1(r) = r^{\frac{1}{2} (n-2\alpha)} J_{\frac{n}{2}}(-i \cdot r)$ and $\varphi_2(r) = r^{\frac{1}{2} (n-2\alpha)}  Y_{\frac{n}{2}}(-i \cdot r)$, where $J_{\nu}, Y_{\nu}$ denote the Bessel functions of first and second kind, respectively, of order $\nu.$
	
	Therefore, all solutions of \eqref{eq:differential} are linear combinations of $\varphi_1$ and $\varphi_2.$ On the other hand, by using the asymptotics  of the Bessel functions $J_{n/2}, Y_{n/2},$ we obtain that
	\begin{align*}
		J_{\alpha}(-iz) \sim \frac{1}{\sqrt{-2\pi i z}} e^{z} \cdot e^{-i(\alpha \pi/2 + \pi/4)}, \, \, Y_{\alpha}(-iz) \sim \frac{-i}{\sqrt{-2\pi i z}} e^{z} \cdot e^{-i(\alpha \pi/2 + \pi/4)},
	\end{align*}
	and thus the only linear combination of $J_{n/2}(-i \cdot r)$ and $Y_{n/2}(-i\cdot r)$ which is integrable on $\R$ is $J_{n/2}(-i \cdot r) - i Y_{n/2}(- i \cdot r) = c_{n} \cdot K_{n/2}(r),$ where $K_{\alpha}$ denotes the modified Bessel function of the second kind.
\end{enumerate}
Hence, the only admissible radial windosw $\psi$ for which $w(t) W_\psi f(-x,t)$ is $\mathcal{M}-$harmonic on the upper half space, for some smooth, positive weight $w,$ are those such that
\begin{equation}\label{eq:admissible-windows}
	\widehat{\overline{\psi}}(r) = c_{n,\alpha} r^{\frac{1}{2}(n-2\alpha)} K_{n/2}(r).
\end{equation}
Relabeling $\frac{n}{2} - \alpha = \beta,$ and using that $\int_{\R^n} \frac{|\widehat{\psi}(\xi)|^2}{|\xi|^n} \, d\xi < +\infty,$ we get the desired condition on $\beta.$ This finishes the proof.
\end{proof}
In spite of this explicit characterization of $\mathcal{M}-$harmonicity, the next main result of this Section shows that, if we restrict our attention to this class of windows, we are \emph{never} able to apply Theorem \ref{thm:main}. That is, the (incidental) fact that if $n=1,$ then the Cauchy wavelets induce a specific class of $\mathcal{M}-$harmonic functions which belongs to the monoid \emph{cannot} happen for dimensions $n\ge2.$ 

\begin{theorem}\label{thm:no-free-lunc-wavelet}
	Let $\psi : \R^n \to \R$ be a radial window of the form \eqref{eq:admissible-windows}. If $n>1,$ then, for each cone $C \subset \R^n$ containing the origin with nonempty interior, there is a function $f \in L^2(\R^n)$ such that $\widehat{f}$ is compactly supported inside $C$ and $\Delta_h(\log |w \cdot W_\psi f|^2) (y,t) < 0,$ for some $y \in \R^n, \, t >0.$
\end{theorem}

\begin{proof} The basic setup of the proof is an argument by contradiction. Suppose that the property holds for each cone $C$ and each $f \in L^2(\R^n)$ which has compact support in $C$. We may suppose, up to a rotation, that $C$ contains the line $\{(t,0,\dots,0), t \in \R\}.$ Take thus a sequence of smooth functions $f_{\varepsilon} \in \mathcal{S}(\R^n)$ such that $\widehat{f_{\varepsilon}}$ is supported inside $C$ and converges as a distribution to $\delta_{e_1} - 2^{\alpha} \delta_{2e_1},$ with $e_i$ denoting the vector whose $j-$th coordinate is $\delta_{i,j}$. Then, since the function $r \mapsto r^{n/2} K_{n/2}(r)$ is smooth on $\R_+,$ and the support of $\widehat{f_{\varepsilon}}$ is disjoint from the origin, we see that
	\begin{align*}
		w(t) W_{\psi} f_{\varepsilon}(-y,t) & = C_{\psi} t^{-\beta} \int_{\R^n} \widehat{f_{\varepsilon}}(\xi) e^{-i y\cdot \xi} \widehat{\overline{\psi}} (t|\xi|) \, d \xi \cr
		& = C_{n,\alpha} \int_{\R^n} \widehat{f_{\varepsilon}}(\xi) |\xi|^{\beta - \frac{n}{2}} e^{-i y \cdot \xi} (t|\xi|)^{\frac{n}{2}} K_{n/2}(t|\xi|) \, d\xi \cr
		& \to C_{n,\alpha} t^{\frac{n}{2}} \left[ K_{n/2}(t) e^{-i y_1} - c'_{n,\alpha} 2^{\frac{1}{2}(n-2\alpha)} K_{n/2}(2t) e^{- 2i y_1} \right], \text{ as } \varepsilon \to 0,
	\end{align*}
	where the convergence stated above holds also pointwise for derivatives of order $\le 2$. Hence, if $\Delta_h (\log |\eta \cdot W_\psi f_{\varepsilon}|^2) \ge 0$ almost everywhere, we readily conclude that
	\[
	\Delta_h [|G(y_1,t)-G(2y_1,2t)|^2] \ge 0, \, \forall \, y_1 \in \R, \, \forall \, t>0,
	\]
	where we used the shorthand $G(y_1,t) := P_n(t) e^{-iy_1}$  where $P_n(t) = t^{n/2} K_{n/2}(t).$ We first note that, for a function $u$ from the upper half-space $\R^{n+1}_+$ to $\R,$ we may compute
	\[
	\Delta_h (\log u) = \frac{t^2}{u^2} \left( u \frac{\Delta_h u}{t^2}  -  (|\nabla_y u|^2 + |\partial_t u|^2)  \right).
	\]
	This last expression is non-negative if, and only if,
	\begin{equation}\label{eq:equiv-positive-laplacian}
		u \frac{\Delta_h u}{t^2}  -  (|\nabla_y u|^2 + |\partial_t u|^2) \ge 0 \text{ almost everywhere in } \R^{n+1}_+.
	\end{equation}
	We wish to use this with $t \to 0$ for $u(y_1,t) = |G(y_1,t) - G(2y_1,2t)|^2.$ To that extent, notice that
	\[
	u(y_1,t) = |P_n(t)|^2 + |P_n(2t)|^2 - 2 P_n(t) P_n(2t) \cos(y_1).
	\]
	Then immediately, we get that $|\nabla_y u| = 2| P_n(t) P_n(2t) \sin(y_1)| \to 2 P_n(0)^2 |\sin(y_1)|$ when $t \to 0,$ as well as $\partial_t u = 2P_n(t) P_n'(t) + 4 P_n(2t) P_n'(2t) - 2 (P_n'(t) P_n(2t) + 2 P_n(t) P_n'(2t)) \cos(y_1) \to 0$ as $t \to 0,$ as the function $P_n$ has a critical point at the origin for $n > 1$ (this follows, for instance, from the Taylor expansion of $K_{n/2}$ around the origin; see \cite[Chapters~9~and~10]{AbrStegun}).
	
	We further see that $u(y_1,t) \to 2 P_n(0)^2(1-\cos(y_1))$ as $t \to 0.$ We now compute the hyperbolic laplacian part: as
	\[
	\frac{\Delta_h u}{t^2} = \Delta_y u + \partial_t^2 u - (n-1)\frac{\partial_t u}{t},
	\]
	we analyze the behaviour of the three terms above separately as $t \to 0.$ For the first, we see easily that
	\[
	\Delta_y u = 2P_n(t)P_n(2t) \cos(y_1) \to 2 P_n(0)^2 \cos(y_1) \text{ as } t \to 0.
	\]
	Moreover,
	\begin{align*}
		\partial_t^2 u & = 2(P_n'(t))^2 + 2 P_n(t) P_n''(t) + 8 (P_n'(2t))^2 + 8 P_n(2t) P_n''(2t) \cr
		& - 2 (2P_n'(t) P_n'(2t) + P_n''(t)P_n(2t) + 2 P_n'(t) P_n'(2t) + 4 P_n(t) P_n''(2t)) \cos(y_1)  \cr
		& \to 10 P_n(0) P_n''(0) - 10 P_n(0) P_n''(0) \cos(y_1) = 10 P_n(0)P_n''(0) (1-\cos(y_1))
	\end{align*}
	as $ t \to 0.$ Finally, note that, since $\partial_t u(y_1,t) \to 0$ as $t \to 0,$ we have $\partial_t u /t \to \partial_t^2 u$ as $t \to 0.$ Hence,
	\[
	\frac{\Delta_h u}{t^2} = 2 P_n(0)^2 \cos(y_1) - 10 (n-2) P_n(0) P_n''(0)(1-\cos(y_1)).
	\]
	We must now distinguish between two different cases: if $n>2,$ the computation above holds in a rigorous level, and no further explaining is needed. If, however, $n=2,$ then we have $P_2''(t) \to - \infty$ as $t \to 0.$ In that case, notice that, from the expansion of $K_1$ around $0$ (cf. \cite{AbrStegun}), there is $a \in \R$ with
	\begin{align*}
	\frac{d^2}{dt^2} \left( P_2(t) - \frac{t^2 \log(t/2)}{2} \right) & \to a \text{ as } t \to 0, \cr
	\frac{1}{t} \cdot \frac{d}{dt} \left( P_2(t) - \frac{t^2 \log(t/2)}{2} \right) & \to a \text{ as } t \to 0. \cr
	\end{align*}
	Hence,
	\[
	P_2''(t) - \frac{1}{t} P_2'(t) \to 1 \text{ as } t \to 0,
 	\]
 	and thus
 	\begin{align*}
 	\partial_t^2 u - \frac{\partial_t u}{t} & = 2(P_2'(t))^2 + 2 P_2(t) P_2''(t) + 8 (P_2'(2t))^2 + 8 P_2(2t) P_2''(2t) \cr
 	& - 2 (2P_2'(t) P_2'(2t) + P_2''(t)P_2(2t) + 2 P_2'(t) P_2'(2t) + 4 P_2(t) P_2''(2t)) \cos(y_1) \cr
 	& -  \frac{1}{t} \left(2P_2(t) P_2'(t) + 4 P_2(2t) P_2'(2t) - 2 (P_2'(t) P_2(2t) + 2 P_2(t) P_2'(2t)) \cos(y_1)  \right) \cr
 	& \to 10P_2(0) (1-\cos(y_1)) \text{ as } t \to 0.
 	\end{align*}
	Plugging these observations into \eqref{eq:equiv-positive-laplacian}, we must have
	\begin{align*}
		2 P_n(0)^2(1-\cos(y_1)) & \left( 2P_n(0)^2 \cos(y_1) - 10(n-2)P_n(0)P_n''(0)(1-\cos(y_1))\right) \cr
		& - 4 P_n(0)^4 (1-\cos^2(y_1)) \ge 0, \text{ for each } y_1 \in \R.
	\end{align*}
    Here, when $n=2,$ $(n-2)P_n''(0)$ is to be interpreted as $-1.$ But this is equivalent to
	\begin{align*}
		2P_n(0)^2 (1-\cos(y_1)) \left(10(n-2)P_n(0) P_n''(0)(1-\cos(y_1)) + 2P_n(0)^2\right) \le 0, \, \text{ for each } y_1 \in \R.
	\end{align*}
	Finally, note that, although $P_n''(0) < 0,$ by choosing $y_1$ sufficiently close to $0,$ we have $10(n-2)P_n(0) P_n''(0)(1-\cos(y_1)) + 2P_n(0)^2 > 0,$ and since $2P_n(0)^2 (1-\cos(y_1))  > 0$ holds for each $y_1 \in \R,$ we reach our desired contradiction.
\end{proof}

As mentioned above, the previous result shows a fundamental difference between dimension $n=1$ and higher dimensions $n\ge 2$. As a matter of fact, by redoing the previous computations for $n=1,$ one obtains the so-called \emph{Cauchy wavelets} ``disguised'' as $t^{\beta} K_{1/2}(t),$ with $\beta > \frac{1}{2},$ and, since the hyperbolic laplacian coincides with the Euclidean one in two dimensions, the sub-harmonicity property is reduced to the sub-harmonicity of the modulus of an analytic function, which is known to hold true by elementary complex analysis, as long as one works within the class of functions which have Fourier support in the positive half line. 

In higher dimensions, however, Theorem \ref{thm:no-free-lunc-wavelet} shows that, even if we restrict the Fourier support of the functions involved to \emph{cones}, in opposition to half spaces, one cannot obtain simultaneously $\mathcal{M}-$harmonicity and $\log-\mathcal{M}-$subharmonicity. 

Thus, Theorem \ref{thm:no-free-lunc-wavelet} shows that, in a way, $\mathcal{M}-$harmonicity and log-subharmonicity are \emph{disjoint} properties for higher-dimensional wavelet transforms, and thus the approach in Theorem \ref{thm:main} fails to give us anything for this particular case. For that reason, we are not able to provide an explicit characterization of extremal sets for concentration of wavelet transforms in higher dimensions, even in the cases given by Proposition \ref{prop:charact-harmonic}. We believe this to be an interesting problem, which remains, to our best knowledge, open.

\section*{Acknowledgements}

J.P.G.R. acknowledges financial support from the ERC grant RSPDE 721675. The authors would like to thank Paolo Tilli for discussions at the early stages of the manuscript.

\end{document}